\documentclass[11pt]{article}

\usepackage[T1]{fontenc}
\usepackage{lmodern}
\usepackage{amsmath,amssymb,amsthm,mathtools}
\usepackage{booktabs}
\usepackage[margin=1in]{geometry}
\usepackage[colorlinks=true,linkcolor=blue,citecolor=blue,urlcolor=blue]{hyperref}

\newtheorem{theorem}{Theorem}[section]
\newtheorem{proposition}[theorem]{Proposition}
\newtheorem{lemma}[theorem]{Lemma}

\theoremstyle{remark}

\newcommand{\tr}{\operatorname{tr}}
\newcommand{\rank}{\operatorname{rank}}

\title{The Exact Maximum of the Spectral Sum of Graphs}
\author{
Jing Huang$^1$ and Wei Wei$^2$\footnote{Email addresses:
jhuangmath@gzhu.edu.cn (J. Huang), weiweimath@sina.com (W. Wei). \newline
Corresponding Author: Jing Huang.}}

\date{\small
1.School of Mathematics and Information Science, Guangzhou University, Guangzhou  510006, China\\
2. Center of Intelligent Computing and Applied Statistics, School of Mathematics, Physics and
Statistics, Shanghai University of Engineering Science, Shanghai 201620, China\\
}

\begin{document}

\maketitle

\begin{abstract}
For a simple graph $G$ of order $n$, let
$S_2(G)=\lambda_1(G)+\lambda_2(G)$ denote its spectral sum.  We
determine, for every $n\geq5$, the exact maximum of $S_2(G)$ and all
equality cases.  The unique maximizer, up to isomorphism, is the
complement of the disjoint union of a suitably balanced complete
bipartite graph and isolated vertices, with the sizes of its three
parts determined by $n$ modulo $7$.  Denoting this graph by
$K_n^\star$, we further show that $ S_2(K_n^\star)\leq\frac{8n}{7}-2,$
with equality exactly when $7\mid n$.  This proves a conjecture of Kumar, Liu,  Monterde,  Pragada  and  Tait, which strengthens the
Aouchiche--Hansen 2010 conjecture  by extending it from connected graphs to
all graphs and by asserting uniqueness of the extremal graph. The
result also subsumes the 2008 conjecture of Ebrahimi B., Mohar,
Nikiforov, and Ahmady.
The proof combines Ky Fan's variational principle with a spectral
inequality for weighted Ferrers quotients to reduce the problem to an
explicit family whose complements have incidence rank one.  Exact
integer optimization and a separate equality analysis then yield the
maximum and uniqueness.
\end{abstract}


\section{Introduction}

All graphs are finite and simple.  For a real symmetric matrix $M$ of
order $n$, write
\[
 \lambda_1(M)\geq\cdots\geq\lambda_n(M),
 \qquad
 S_k(M)=\sum_{i=1}^k\lambda_i(M)
 \quad (1\leq k\leq n).
\]
For a graph $G$, write $S_k(G)=S_k(A(G))$.  Let $I_n$ and $J_n$
denote the identity matrix and the all-ones matrix of order $n$,
respectively, and omit the subscripts when the order is clear.
Both the largest and second largest adjacency eigenvalues have been
studied extensively; see the surveys
\cite{CvetkovicRowlinson1990,CvetkovicSimic1995} and the monograph
\cite{Stanic2015}.  The sum $S_2(G)$ is the first nontrivial partial
sum of the adjacency spectrum.  Whereas the spectral radius is
characterized by a single Rayleigh quotient, Ky Fan's principle
\cite{Fan1949} expresses $S_2(G)$ as the maximum Rayleigh trace over
orthonormal pairs.  The partial-sum viewpoint also connects $S_2$ with
graph energy:  $
 \mathcal E(G)=\sum_i|\lambda_i(G)|
              =2\max_{1\leq k\leq n}S_k(G);$
see \cite{Gutman1978,Nikiforov2016}.

General bounds for partial spectral sums were obtained by Mohar
\cite{Mohar2009} and subsequently refined using inertia and rank by
Das, Mojallal and Sun \cite{DasEtAl2019}.  Related work treats random
graphs \cite{Rocha2020}, symmetric and nonnegative matrices
\cite{Kolotilina2013,SunMinDasMatrices2026}, and exact extremal or
classification questions for trees and for connected graphs with small
spectral sum
\cite{KumarTrees2026,SunMinDasExtremal2026}.  These results do not
determine the unrestricted maximum at a fixed order.

The exact maximum over all graphs of a prescribed order has a separate
history.  Gernert conjectured that every graph $G$ of order $n$
satisfies $S_2(G)\leq n$, and verified the inequality for regular
graphs and several other graph classes; see
\cite{EbrahimiEtAl2008}.  Nikiforov  \cite{Nikiforov2006} subsequently disproved this conjecture
and established linear upper and lower bounds for the unrestricted
maximum.  A more precise conjectural picture
emerged from the work of Ebrahimi B., Mohar, Nikiforov and
Ahmady \cite{EbrahimiEtAl2008}.  Besides sharpening the general upper
bound, they showed that, for $n=7k$, the graph obtained from $K_{7k}$
by deleting the edges of $K_{2k,2k}$ has spectral sum
$8n/7-2.$
This construction shows that the leading coefficient in any universal
linear upper bound cannot be smaller than $8/7$, and led them to
conjecture that $ S_2(G)\leq 8n/7$
for every graph $G$.
The construction above belongs to a natural three-part family that
provides candidates at every order.  For $n\geq p+q$, define
\[
 K(n,p,q)=\overline{K_{p,q}\cup(n-p-q)K_1}
           =K_{n-p-q}\vee(K_p\cup K_q).
\]
We call the vertex sets of sizes $p$, $q$, and $n-p-q$ the
$p$-part, the $q$-part, and the universal part, respectively.
To specify the candidate of order $n$, write $n=7k+r$, where
$0\leq r\leq6$, and put
\begin{equation}\label{eq:sn-table}
 s_n=
 \begin{cases}
 4k,&r=0,1,\\
 4k+1,&r=2,\\
 4k+2,&r=3,4,\\
 4k+3,&r=5,\\
 4k+4,&r=6.
 \end{cases}
\end{equation}
Define
\begin{equation}\label{eq:extremal-candidate}
 K_n^\star=K\left(n,
  \left\lceil\frac{s_n}{2}\right\rceil,
  \left\lfloor\frac{s_n}{2}\right\rfloor\right).
\end{equation}
Thus the two nonadjacent parts differ in size by at most one, while
their total size is $4n/7+O(1)$.
With this notation, Aouchiche and Hansen
\cite[Conjecture~17]{AouchicheHansen2010} proposed $K_n^\star$ as the
residue-dependent extremal graph for connected graphs.  Their
formulation asserts that $K_n^\star$ maximizes $S_2$ and that
$ S_2(K_n^\star)\leq\frac{8n}{7}-2,$
with equality exactly when $7\mid n$, but does not include a
uniqueness clause.
Recently, Kumar, Liu, Monterde, Pragada and Tait  \cite[Theorem~1.2]{KumarEtAl2026} proved the universal
bound $ S_2(G)\leq8n/7$
for every graph $G$.  Together with
the construction given in \cite{EbrahimiEtAl2008},
this established $8/7$ as the optimal linear coefficient but left the
residue-dependent finite-order maximum and its equality cases open.
The authors subsequently recorded the remaining exact finite-order
problem, for all graphs and with uniqueness, as
Conjecture~5.1 in \cite{KumarEtAl2026}: $S_2(G)\leq S_2(K_n^\star)$ for
every graph $G$ of order $n\geq5$, with equality if and only if
$G\cong K_n^\star$.
Our main result, stated below, fully resolves this conjecture along with its sharp finite-order bound.

\begin{theorem}\label{thm:main}
For every graph $G$ of order $n\geq5$, we have
\begin{equation}\label{eq:main-bound}
 S_2(G)\leq S_2(K_n^\star).
\end{equation}
Equality in \eqref{eq:main-bound} holds if and only if
$G\cong K_n^\star$.  Moreover, $S_2(K_n^\star)\leq\frac{8n}{7}-2,$
with equality if and only if $7\mid n$.
\end{theorem}

The proof separates the determination of the extremal value from the
uniqueness argument.  Choose a maximizer with as many edges as
possible.  Ky Fan's variational principle turns edge-maximality into a
threshold rule, which forces the complement to be a chain graph.  A new
sign-variation argument yields a spectral-tail inequality for the
associated weighted Ferrers quotient and, in turn, a uniform defect
whenever the incidence rank is at least two.  Since incidence rank zero
gives the complete graph, which is suboptimal, the complement must have
incidence rank one.  The maximizer therefore belongs to the family
$K(n,p,q)$, and exact integer optimization selects $K_n^\star$.
Finally, a separate analysis of equality in the threshold rule rules
out all other maximizers.

The remainder of the paper is organized in two sections.
Section~\ref{sec:basic} collects the variational and structural tools,
optimizes the family $K(n,p,q)$, and reduces an edge-maximal extremizer
to one whose complement is a chain graph.  Section~\ref{sec:technical}
establishes the Ferrers tail inequality and the uniform defect for
incidence rank at least two, and then combines these estimates with
the rank-one optimization to prove Theorem~\ref{thm:main}, including
uniqueness.

\section{Basic results and structural tools}
\label{sec:basic}

We first record Ky Fan's variational principle and optimize the
three-part family $K(n,p,q)$.  We then use the same principle to obtain
the chain-graph reduction used in the proof of the main theorem.

\begin{lemma}\label{lem:ky-fan}
Let $M$ be a real symmetric matrix of order $n\geq2$.  Then
\[
 S_2(M)=\max\bigl\{
 \mathbf x^{\mathsf T}M\mathbf x+
 \mathbf y^{\mathsf T}M\mathbf y:
 \mathbf x,\mathbf y\in\mathbb R^n
 \text{ are orthonormal}
 \bigr\}.
\]
If $\lambda_1(M)$ is simple, equality holds for
$\mathbf x,\mathbf y$ if and only if
$\operatorname{span}\{\mathbf x,\mathbf y\}$ is spanned by an
eigenvector for $\lambda_1(M)$ and an eigenvector for
$\lambda_2(M)$.
\end{lemma}

\begin{proof}
Let $\mathbf v_1,\ldots,\mathbf v_n$ be an orthonormal eigenbasis of
$M$, with corresponding eigenvalues
$\lambda_1\geq\cdots\geq\lambda_n$, and let
$P=\mathbf x\mathbf x^{\mathsf T}
  +\mathbf y\mathbf y^{\mathsf T}.$
Then $P$ is the orthogonal projector onto
$\operatorname{span}\{\mathbf x,\mathbf y\}$.  Put
$c_i=\|P\mathbf v_i\|^2$.  Since $0\leq c_i\leq1$ and
$\sum_i c_i=\tr P=2$,
\[
\begin{aligned}
 \mathbf x^{\mathsf T}M\mathbf x+
 \mathbf y^{\mathsf T}M\mathbf y
 &=\tr(PM)=\sum_{i=1}^n\lambda_i c_i\\
 &\leq\lambda_1c_1+\lambda_2(2-c_1)
 \leq\lambda_1+\lambda_2.
\end{aligned}
\]
Taking $\mathbf x=\mathbf v_1$ and $\mathbf y=\mathbf v_2$ proves the
variational formula.
Suppose that $\lambda_1$ is simple.  Then
$\lambda_1>\lambda_2$, so equality in the last inequality forces
$c_1=1$, equivalently
$\mathbf v_1\in\operatorname{span}\{\mathbf x,\mathbf y\}$.
Equality in the preceding inequality forces the remaining line of
this span to lie in the eigenspace of $\lambda_2$.  The converse is
immediate.
\end{proof}

For a graph $G$ of order $n$, set
\[
 B(G)=A(G)+I_n,
 \qquad S_2(B(G))=S_2(G)+2.
\]

\begin{lemma}\label{lem:candidate-poly}
Let $p,q,c$ be positive integers with $p+q+c=n$. Then $B(K(n,p,q))$ has rank three, and its nonzero eigenvalues are the roots of $ t^3-nt^2+pqt+cpq=0.$
Two are positive and the third is $-\tau$ for some $\tau>0$. Moreover,
\[
 S_2(K(n,p,q))=n-2+\tau, \quad \tau^3+n\tau^2+pq\tau=cpq.
\]
\end{lemma}

\begin{proof}
The partition into the three parts of sizes $p,q,c$ is equitable for $B(K(n,p,q))$, with quotient matrix
\[
 Q= \begin{pmatrix} p&0&c\\ 0&q&c\\ p&q&c \end{pmatrix}.
\]
A calculation gives $\det(tI_3-Q)=t^3-nt^2+pqt+cpq$ and $\det Q=-pqc\neq0$.
Since $B(K(n,p,q))$ has three row types, its rank is at most $3$. Because $\det Q\neq0$, its rank is exactly $3$, and its nonzero eigenvalues are precisely those of $Q$.

These eigenvalues are real, have product $\det Q=-pqc<0$, and sum $\tr Q=n>0$. Hence two are positive and one is negative; write the latter as $-\tau$ ($\tau>0$). The two positive eigenvalues are the two largest eigenvalues of $B(K(n,p,q))$, so
\[
 S_2(K(n,p,q)) = S_2(B(K(n,p,q))) - 2 = n-2+\tau.
\]
Substitution of $t=-\tau$ into $t^3-nt^2+pqt+cpq=0$  yields $\tau^3+n\tau^2+pq\tau=cpq$.
\end{proof}

We next determine the exact integer parameters $(p,q)$ that maximize this spectral sum.
\begin{proposition}\label{prop:integer-opt}
For every $n\geq5$, the maximum of $S_2(K(n,p,q))$ over integers $p\geq q\geq0$ with $p+q\leq n$ is attained uniquely at
\[
 (p,q)=\left(\left\lceil\frac{s_n}{2}\right\rceil, \left\lfloor\frac{s_n}{2}\right\rfloor\right).
\]
\end{proposition}

\begin{proof}
If $q=0$ or $p+q=n$, then $S_2(K(n,p,q))=n-2$. By \eqref{eq:sn-table}, all three parts of $K_n^\star$ are nonempty, so Lemma~\ref{lem:candidate-poly} gives $S_2(K_n^\star)>n-2$. Hence no boundary pair can maximize, and we restrict to $p\geq q>0$ with $p+q<n$.

Fix $s=p+q$. By Lemma~\ref{lem:candidate-poly},
$S_2(K(n,p,q))=n-2+\tau$, where
$\tau^3+n\tau^2=pq(n-s-\tau)$, implying $\tau<n-s$.
For two feasible pairs with sum $s$, let $a=pq<a'=p'q'$, and let
$\tau$ and $\tau'$ be the corresponding positive roots.  At
$t=\tau$, the polynomial corresponding to $a'$ has value
\[
 \tau^3+n\tau^2+a'\tau-(n-s)a'
 =-(a'-a)(n-s-\tau)<0.
\]
This polynomial is strictly increasing on $[0,\infty)$ and vanishes
at $\tau'$, so $\tau'>\tau$.  Thus $S_2$ increases strictly with
$pq$, uniquely maximizing at
$(p,q)=(\lceil s/2\rceil,\lfloor s/2\rfloor)$.
It remains to maximize the corresponding positive root $\tau_s$ of
\begin{equation}\label{eq:balanced-root}
 t^3+nt^2+\left\lfloor\frac{s^2}{4}\right\rfloor t -(n-s)\left\lfloor\frac{s^2}{4}\right\rfloor=0
\end{equation}
for $2\leq s\leq n-1$. Set $L_s=n-s-\lfloor s/2\rfloor-1$. Evaluating the left-hand side of \eqref{eq:balanced-root} for $s+1$ at $t=\tau_s$ yields $\lfloor(s+1)/2\rfloor(\tau_s-L_s)$. Since the polynomials are strictly increasing, $\tau_{s+1}>\tau_s$ if and only if $\tau_s<L_s$, and $\tau_{s+1}<\tau_s$ if and only if $\tau_s>L_s$.
When $L_s>0$, write $s=2h$ or $2h+1$. Substituting $t=L_s$ into \eqref{eq:balanced-root} gives $(n-2h-1)R_s$, where
\[
 R_{2h}=14h^2-11hn+7h+2n^2-3n+1, \quad R_{2h+1}=14h^2-11hn+21h+2n^2-8n+8.
\]
Since $n-2h-1>0$, we obtain $R_s>0$ exactly when
$\tau_s<L_s$, and $R_s<0$ exactly when $\tau_s>L_s$.
If $L_{2h}>0$, then $n\geq3h+2$, and hence $
 R_{2h+1}-R_{2h} =14h-5n+7<0.$
If $L_{2h+1}>0$, then $n\geq3h+3$, and hence  $
 R_{2h+2}-R_{2h+1}=14h-6n+14<0.$
Thus $R_s$ strictly decreases whenever $L_s>0$.
For $n=7k+r$, direct substitution gives
\[
\begin{array}{c|ccccccc}
 r & 0 & 1 & 2 & 3 & 4 & 5 & 6 \\ \hline
 R_{s_n-1} & 7k+1 & 13k+6 & 5k+3 & 4k+2 & 10k+8 & 2k+2 & k+1 \\
 R_{s_n} & 1-7k & -k & -2k & -10k-2 & -4k-2 & -5k-2 & -13k-7
\end{array}
\]
Moreover, \eqref{eq:sn-table} gives
$L_{s_n-1}>0$, and the displayed table shows that
$R_{s_n-1}>0>R_{s_n}$ for every $n\geq5$.
If $2\leq s<s_n$, $L_s$ decreases strictly, so $L_s>0$. Monotonicity gives $R_s\geq R_{s_n-1}>0$, which forces $\tau_{s+1}>\tau_s$.
If $s_n\leq s\leq n-2$, we consider two cases: if $L_s>0$, then $R_s\leq R_{s_n}<0$, yielding $\tau_s>L_s$; if $L_s\leq0$, then $\tau_s>0\geq L_s$. Both cases ensure $\tau_s>L_s$, giving $\tau_{s+1}<\tau_s$.
Therefore, $\tau_s$ is uniquely maximized at $s=s_n$.
\end{proof}

Let $-\tau_n$ be the negative eigenvalue of $B(K_n^\star)$.  Then
\begin{equation}\label{eq:candidate-value}
 S_2(K_n^\star)=n-2+\tau_n.
\end{equation}
Having defined the maximal candidate value in terms of $\tau_n$, we now establish sharp analytical bounds on this parameter.
\begin{lemma}
For every $n\geq5$, we have
\begin{equation}\label{eq:tau-envelope}
 0<\tau_n\leq\frac n7.
\end{equation}
The upper equality holds if and only if $7\mid n$. Moreover, for $f_n(t)=\frac34t^2+(t-n)^2$,
\begin{equation}\label{eq:candidate-deficit}
 0\leq n^2-f_n(n+\tau_n)<\frac12.
\end{equation}
\end{lemma}

\begin{proof}
Set $m=\left\lfloor\frac{s_n^2}{4}\right\rfloor$ and $F(t)=t^3+nt^2+mt-(n-s_n)m$. By Lemma~\ref{lem:candidate-poly}, $F(\tau_n)=0$, and $F$ is strictly increasing on $[0,\infty)$. Writing $n=7k+r$ and using \eqref{eq:sn-table} gives
\[
\begin{array}{c|ccccccc}
 r & 0 & 1 & 2 & 3 & 4 & 5 & 6 \\ \hline
 343F(n/7) & 0 & 168k+8 & 182k+64 & 42k+20 & 42k+22 & 182k+118 & 168k+160
\end{array}
\]
Thus $F(n/7)\geq0$, with equality exactly when $r=0$. Since $F(\tau_n)=0$ and $F$ is strictly increasing, this proves \eqref{eq:tau-envelope} and its equality statement.
The same table gives $F(n/7)\leq\frac{26n+12}{343}$. Expanding at $n/7$ and using the preceding bound gives
\begin{align*}
 F\left(\frac n7-\frac1{4n}\right) &= F\left(\frac n7\right) - \frac1{4n}\left(\frac{17n^2}{49}+m\right) + \frac5{56n} - \frac1{64n^3} < 0.
\end{align*}
Since $F$ is strictly increasing and $F(\tau_n)=0$, we obtain $0\leq \frac n7-\tau_n < \frac1{4n}$.
Finally, \eqref{eq:tau-envelope} gives
\[
 0\leq n^2-f_n(n+\tau_n) = \frac74(n+\tau_n)\left(\frac n7-\tau_n\right) < \frac12.
\]
This proves \eqref{eq:candidate-deficit}.
\end{proof}

We conclude our analysis of the candidate family by showing that the spectral sum
of $K_n^\star$ strictly increases with its order.
\begin{lemma}
\label{lem:strict-growth}
For every $n\geq6$, $S_2(K_n^\star)>S_2(K_{n-1}^\star).$
\end{lemma}

\begin{proof}
By \eqref{eq:sn-table} and \eqref{eq:extremal-candidate}, passing from
$K_{n-1}^\star$ to $K_n^\star$ increases the universal part by one
when $n\equiv0,1,4\pmod7$, the $p$-part by one when
$n\equiv2,5\pmod7$, and the $q$-part by one when
$n\equiv3,6\pmod7$; the other two parts remain unchanged.
Consequently, $K_{n-1}^\star$ is an induced subgraph of
$K_n^\star$, and $B(K_{n-1}^\star)$ is a proper principal submatrix of
$B(K_n^\star)$.
The universal part of $K_n^\star$ has size $n-s_n>0$, so $K_n^\star$ is
connected and $B(K_n^\star)$ is irreducible and nonnegative.
Perron--Frobenius and Cauchy interlacing give
\[
 \lambda_1(B(K_n^\star))
 >\lambda_1(B(K_{n-1}^\star))
 \geq\lambda_2(B(K_n^\star))
 \geq\lambda_2(B(K_{n-1}^\star)).
\]
Adding the first and last comparisons and using
$S_2(B(G))=S_2(G)+2$ proves the desired result.
\end{proof}

The preceding lemmas identify the optimal candidate in the family
$K(n,p,q)$.  We now turn to the structural reduction that explains why
this family is extremal.  The natural parameter in this reduction is
the incidence rank of the complement.
For a bipartite graph $H$ with bipartition $U\cup V$, its
\emph{bipartite incidence matrix} is the $0$--$1$ matrix whose rows
are indexed by $U$, whose columns are indexed by $V$, and whose
$(u,v)$-entry is one if and only if $uv\in E(H)$. The rank of this matrix over $\mathbb R$ is called the
\emph{incidence rank} of $H$.
The bipartite graph $H$ is a \emph{chain graph} if the vertices of $U$
can be ordered as $u_1,\ldots,u_p$ so that
\[
 N_H(u_1)\subseteq N_H(u_2)\subseteq\cdots\subseteq N_H(u_p).
\]
Equivalently, the vertices in the two bipartition classes can be
ordered so that every row of the bipartite incidence matrix consists
of an initial block of ones followed by zeros, with the block lengths
nondecreasing down the rows.  Such a matrix is called a
\emph{Ferrers matrix}.  We use this ordering throughout.
The following lemma provides a sufficient algebraic condition for a graph to have a chain graph as its complement.
\begin{lemma}\label{lem:threshold-complement}
Let $B$ be a symmetric $0$--$1$ matrix with diagonal entries one.
Suppose that there exist vectors $\mathbf x>0$ and $\mathbf y$ such
that, for $u\ne v$, $B_{uv}=1$ exactly when $
 x_ux_v+y_uy_v\geq0.$
Then the graph with adjacency matrix $J-B$ is a chain graph.
\end{lemma}

\begin{proof}
Put $r_u=y_u/x_u$.  Since $x_u>0$, for $u\ne v$,
$B_{uv}=0$ exactly when $1+r_ur_v<0.$
Thus every edge of the graph with adjacency matrix $J-B$ joins a
vertex with positive $r_u$ to one with negative $r_u$, while vertices
with $r_u=0$ are isolated.  If $r_u,r_v$ have the same nonzero sign,
$|r_u|\leq|r_v|$, and $w$ is adjacent to $u$, then $r_w$ has the
opposite sign and $r_vr_w\leq r_ur_w<-1.$
Hence every neighbor of $u$ is also a neighbor of $v$.  Ordering the
vertices on each side by $|r_u|$ therefore orders their neighborhoods
by inclusion, so the graph is a chain graph.
\end{proof}

We now apply this threshold condition to establish the main structural reduction for an edge-maximal extremal graph.
\begin{proposition}\label{prop:ferrers-reduction}
Let $G$ be an $n$-vertex graph maximizing $S_2(G)$, chosen among all
such graphs to have the maximum number of edges.  Then $G$ is
connected and $\overline G$ is a chain graph.
\end{proposition}

\begin{proof}
Let $\mathbf x,\mathbf y$ be any orthonormal eigenvectors of $A(G)$
corresponding to $\lambda_1(G)$ and $\lambda_2(G)$, with vertex
coordinates $x_u,y_u$.

If $uv\notin E(G)$, Lemma~\ref{lem:ky-fan} gives
$S_2(G+uv)\geq S_2(G)+2(x_ux_v+y_uy_v)$.  Hence
$x_ux_v+y_uy_v\leq0$.  Equality would produce a maximizer with more
edges, contrary to the choice of $G$, so the inequality is strict.
Similarly, deleting an edge $uv$ gives
$S_2(G-uv)\geq S_2(G)-2(x_ux_v+y_uy_v)$, and hence
$x_ux_v+y_uy_v\geq0$.  Thus $uv\in E(G)$ exactly when
$x_ux_v+y_uy_v\geq0$.

Suppose that $G$ is disconnected.  Since $A(G)$ is block diagonal and
the preceding argument applies to any such pair of eigenvectors, we
may choose each eigenvector to be supported on one component.  If
their supporting components are distinct, choose one vertex from
each; otherwise choose one vertex inside their common supporting
component and one outside it.  In either case, we obtain vertices
$u,v$ in distinct components with $x_ux_v+y_uy_v=0$.  The preceding
criterion would make $uv$ an edge, a contradiction.  Hence $G$ is
connected.
By the Perron--Frobenius theorem, we may now choose $\mathbf x>0$.
The same criterion and Lemma~\ref{lem:threshold-complement}, applied
to $B(G)$, show that $\overline G$ is a chain graph.
\end{proof}

We now organize the argument by the incidence rank of $\overline G$.
Rank zero means that $\overline G$ is empty, so $G=K_n=K(n,0,0)$.
If the incidence rank is one, all nonzero rows of the bipartite
incidence matrix are equal: indeed, they are scalar multiples of one
another, and two nonzero $0$--$1$ rows can be scalar multiples only
with scalar one.  Hence $\overline G$ is a complete bipartite graph
together with isolated vertices, and therefore $ G\cong K(n,p,q)$
for some $p,q>0$ with $p+q\leq n$.  Proposition~\ref{prop:integer-opt}
therefore settles the cases of incidence rank at most one.  It remains
to exclude incidence rank at least two, which is the purpose of the
technical estimates developed in the next section.

\section{Technical estimates and proof of the main theorem}
\label{sec:technical}

The exclusion of chain complements of incidence rank at least two
rests on two estimates: a spectral-tail inequality for weighted
Ferrers quotients and a uniform numerical defect.  To prove the first,
we compress a chain graph to its nonzero neighborhood classes and use
the sign pattern of the inverse quotient to control coefficient sign
changes.

For a finite sequence $(a_0,\ldots,a_m)$ with no zero terms, let
$\operatorname{var}(a_0,\ldots,a_m)$ be the number of indices
$j\in\{1,\ldots,m\}$ such that $a_{j-1}a_j<0$.  These indices are
called the sign-change positions.
A real matrix is \emph{totally nonnegative} if all its minors are
nonnegative.  The following is the zero-free specialization of the
variation-diminishing theorem
\cite[Theorem~3.4(a)]{Pinkus2010}.

\begin{lemma}
\label{lem:variation-diminishing}
Let $M$ be totally nonnegative.  If $\mathbf x$ and $M\mathbf x$ have
no zero coordinates, then $
 \operatorname{var}(M\mathbf x)\leq\operatorname{var}(\mathbf x).$
\end{lemma}

We now use this property to bound the displacement of sign-change positions.
\begin{lemma}
\label{lem:variation-location}
Let $\alpha>0$ and
\[
 P(t)=\sum_{j=0}^{m}a_jt^j,
 \qquad
 Q(t)=(1+\alpha t)P(t)=\sum_{j=0}^{m+1}b_jt^j
\]
have no zero coefficients.  If the sign-change positions of their
coefficient sequences are $ i_1<\cdots<i_r$ and $j_1<\cdots<j_r,$
respectively, then $i_\ell\leq j_\ell\leq i_\ell+1$ for $1\leq\ell\leq r.$
\end{lemma}

\begin{proof}

Comparing coefficients in $Q(t)=(1+\alpha t)P(t)$ gives
\[
 b_0=a_0,\qquad
 b_j=a_j+\alpha a_{j-1}\quad(1\leq j\leq m),\qquad
 b_{m+1}=\alpha a_m.
\]
Then, for each  $0\leq p\leq m$,
\[
 \begin{pmatrix}
  b_0\\
  b_1\\
  \vdots\\
  b_p
 \end{pmatrix}
 =
 \begin{pmatrix}
  1      &        &        &   \\
  \alpha & 1      &        &   \\
         & \ddots & \ddots &   \\
         &        & \alpha & 1
 \end{pmatrix}
 \begin{pmatrix}
  a_0\\
  a_1\\
  \vdots\\
  a_p
 \end{pmatrix}, \qquad
 \begin{pmatrix}
  b_{m+1}\\
  b_m\\
  \vdots\\
  b_{p+1}
 \end{pmatrix}
 =
 \begin{pmatrix}
  \alpha &        &        &   \\
  1      & \alpha &        &   \\
         & \ddots & \ddots &   \\
         &        & 1      & \alpha
 \end{pmatrix}
 \begin{pmatrix}
  a_m\\
  a_{m-1}\\
  \vdots\\
  a_p
 \end{pmatrix}.
\]
Every minor of either displayed matrix is either zero or a
nonnegative power of $\alpha$, so both matrices are totally
nonnegative.
Lemma~\ref{lem:variation-diminishing}, together with the fact that
reversing a sequence does not change its variation, therefore gives
\begin{align}
 \operatorname{var}(b_0,\ldots,b_p)
 &\leq\operatorname{var}(a_0,\ldots,a_p),
 \label{eq:variation-prefix}\\
 \operatorname{var}(b_{p+1},\ldots,b_{m+1})
 &\leq\operatorname{var}(a_p,\ldots,a_m).
 \label{eq:variation-suffix}
\end{align}
Since both full coefficient sequences have variation $r$, additivity
of variation across a common endpoint and
\eqref{eq:variation-suffix} yield
\begin{equation}
\label{eq:variation-shifted-prefix}
\begin{aligned}
 \operatorname{var}(b_0,\ldots,b_{p+1})
 =r-\operatorname{var}(b_{p+1},\ldots,b_{m+1})\geq \operatorname{var}(a_0,\ldots,a_p).
\end{aligned}
\end{equation}

For every $p<i_\ell$, the definition of $i_\ell$ and
\eqref{eq:variation-prefix} give $
 \operatorname{var}(b_0,\ldots,b_p)
 \leq\operatorname{var}(a_0,\ldots,a_p)<\ell.$
If $j_\ell<i_\ell$, then \eqref{eq:variation-prefix} with
$p=j_\ell$ gives $
 \ell
 =\operatorname{var}(b_0,\ldots,b_{j_\ell})
 \leq\operatorname{var}(a_0,\ldots,a_{j_\ell})
 <\ell,$
a contradiction.  Hence $j_\ell\geq i_\ell$.  On the other hand, taking
$p=i_\ell$ in \eqref{eq:variation-shifted-prefix} gives
\[
 \operatorname{var}(b_0,\ldots,b_{i_\ell+1})
 \geq\operatorname{var}(a_0,\ldots,a_{i_\ell})
 =\ell.
\]
Hence the $\ell$-th sign change of $Q$ occurs no later than position
$i_\ell+1$, so $j_\ell\leq i_\ell+1$.
\end{proof}

Let $H$ be a bipartite chain graph of incidence rank $k\geq1$.
Deleting zero rows and zero columns from a bipartite incidence matrix
does not change its rank.  Since $H$ is a chain graph, its distinct
nonzero row supports are linearly ordered by inclusion; write
$\varnothing=R_0\subsetneq R_1\subsetneq\cdots\subsetneq R_\ell.$
For each $1\leq i\leq\ell$, choose
$v_i\in R_i\setminus R_{i-1}$.  The submatrix whose rows are the
distinct rows with supports $R_1,\ldots,R_\ell$ and whose columns are
$v_1,\ldots,v_\ell$ has $(i,j)$-entry one if and only if $j\leq i$.
It is therefore lower triangular with diagonal entries one, so these
$\ell$ rows are linearly independent.  Since every nonzero row is
equal to one of them, they also span the row space.  Hence
$\ell=k$, and thus $
 \varnothing=R_0\subsetneq R_1\subsetneq\cdots\subsetneq R_k.$
For $1\leq i\leq k$, let $p_i$ be the number of rows with support
$R_i$, and put $q_i=|R_i\setminus R_{i-1}|.$
 Define
\[
 C=\bigl(\sqrt{p_iq_j}\,\mathbf1_{\{j\leq i\}}\bigr)_
       {1\leq i,j\leq k},  \quad W=
 \begin{pmatrix}
 {\bf  0}&C\\
  C^{\mathsf T}& {\bf  0}
 \end{pmatrix},
 \quad
 \mathbf x=
 \bigl(\sqrt{p_1},\ldots,\sqrt{p_k},
       \sqrt{q_1},\ldots,\sqrt{q_k}\bigr)^{\mathsf T}.
\]
Write $c\geq0$ for the number of isolated vertices of $H$.  Set
\begin{equation}\label{eq:ferrers-quotient}
 M_0=\mathbf x\mathbf x^{\mathsf T}-W,
 \qquad
 M_c=
 \begin{pmatrix}
  M_0&\sqrt c\,\mathbf x\\
  \sqrt c\,\mathbf x^{\mathsf T}&c
 \end{pmatrix}
 \quad(c>0).
\end{equation}

With these matrices defined, we now show that the spectrum of $J-A(H)$ reduces to that of the smaller quotient matrix $M_0$ or $M_c$.
\begin{lemma}\label{lem:ferrers-quotient}
Let $H$ be a bipartite chain graph of incidence rank $k\geq1$.  If $c=0$, the spectrum of $J-A(H)$ consists of
the spectrum of $M_0$ together with zeros.  If $c>0$, it consists of
the spectrum of $M_c$ together with zeros.
\end{lemma}

\begin{proof}
Partition the vertices of $H$ into the row blocks of sizes
$p_1,\ldots,p_k$, the column blocks of sizes
$q_1,\ldots,q_k$, and, when $c>0$, the block of isolated vertices.
Let $\mathcal U\subseteq\mathbb R^{V(H)}$ be the subspace of vectors
that are constant on each block.  For each block of size $d$, take the
vector that equals $1/\sqrt d$ on that block and zero elsewhere.
Taking these vectors in the order just given produces an orthonormal
basis of $\mathcal U$.  Since every block of $A(H)$ and $J$ is
constant, both matrices leave $\mathcal U$ invariant.

The block of $A(H)$ between the $i$-th row block and the $j$-th column
block is an all-ones block if $j\leq i$ and a zero block otherwise.
Consequently, the corresponding entry of the matrix of
$A(H)|_{\mathcal U}$ is
$\sqrt{p_iq_j}\,\mathbf1_{\{j\leq i\}}=C_{ij}.$
All row--row and column--column blocks are zero, as are all blocks
involving the isolated vertices.  Hence the nonisolated part of this
matrix is $W$.

For two blocks of sizes $d$ and $e$, the corresponding entry of the
matrix of $J|_{\mathcal U}$ is $\sqrt{de}$.  Thus this matrix is the
outer product of the vector of square roots of the block sizes with
itself.  That vector is $\mathbf x$ when $c=0$ and
$\binom{\mathbf x}{\sqrt c}$ when $c>0$.  It follows from
\eqref{eq:ferrers-quotient} that the matrix of
$(J-A(H))|_{\mathcal U}$ is $M_0$ when $c=0$ and $M_c$ when $c>0$.

Now let $\mathbf v\in\mathcal U^\perp$.  Orthogonality to the
normalized indicator vector of each block implies that the coordinates
of $\mathbf v$ sum to zero on that block.  Since every row of $A(H)$
and $J$ is constant on each block, every coordinate of
$A(H)\mathbf v$ and $J\mathbf v$ is a linear combination of these
block sums.  Therefore
$ A(H)\mathbf v=J\mathbf v=0,$ which leads to
$ \mathcal U^\perp\subseteq\ker(J-A(H)).$
Complete the orthonormal basis of $\mathcal U$ above by an orthonormal
basis of $\mathcal U^\perp$.  With respect to the resulting basis,
\[
J-A(H)=
 \begin{cases}
  M_0\oplus {\bf  0},&c=0,\\
  M_c\oplus {\bf  0},&c>0.
 \end{cases}
\]
The assertion follows.
\end{proof}

For a polynomial $P(t)=\sum_{j=0}^m a_jt^j$ with no zero
coefficients, write $\operatorname{var}(P)=\operatorname{var}(a_0,\ldots,a_m).$

\begin{lemma}
\label{lem:real-rooted-variation}
Let $P$ be real-rooted and have no zero root or zero coefficient.
Then $\operatorname{var}(P)$ equals the number of positive roots of
$P$, counted with multiplicity.
\end{lemma}

\begin{proof}
The coefficient of $t^j$ in $P(-t)$ is $(-1)^ja_j$.  Thus, at each
position, exactly one of the coefficient sequences of $P(t)$ and
$P(-t)$ has a sign change.  Hence
\[
 \operatorname{var}(P)+\operatorname{var}(P(-t))=\deg P.
\]
Descartes' rule of signs
\cite[Theorem]{Wang2004}  bounds the numbers of positive roots of
$P(t)$ and $P(-t)$ by these two variations, respectively.  These are
the numbers of positive and negative roots of $P$.  Since all roots
of $P$ are real and nonzero, their sum is $\deg P$, so both bounds are
equalities.
\end{proof}

The \emph{inertia} of a real symmetric matrix is the triple
$(n_+,n_-,n_0)$ giving the numbers of its positive, negative, and zero
eigenvalues, counted with multiplicity.  In the next lemma, the parameters $p_i,q_i,c$ in \eqref{eq:ferrers-quotient} are allowed to be arbitrary positive real
numbers.
\begin{lemma}\label{lem:positive-ferrers-tail}
Let $k\geq2$, and let
$\lambda_1\geq\cdots\geq\lambda_{2k+1}$ be the eigenvalues of $M_c$.
Then $\sum_{3\leq i<j\leq2k+1}\lambda_i\lambda_j\leq0.$
Moreover, $M_0$ has inertia $(k+1,k-1,0)$.
\end{lemma}

\begin{proof}
The matrix $C$ is lower triangular with diagonal entries
$\sqrt{p_iq_i}$, and is therefore invertible.  Hence $W$ is congruent
to $\left(\begin{smallmatrix} {\bf  0}&I_k\\ I_k& {\bf  0}\end{smallmatrix}\right)$
and has inertia $(k,k,0)$.  The Schur complement of the bottom-right
entry $c$ of $M_c$ is
$\mathbf x\mathbf x^{\mathsf T}-W-
c^{-1}(\sqrt c\,\mathbf x)(\sqrt c\,\mathbf x)^{\mathsf T}=-W$.
Thus $M_c$ has inertia $(k+1,k,0)$, and
\begin{equation}\label{eq:ferrers-det}
 \det M_c=c\det(-W)=(-1)^kc\prod_{i=1}^kp_iq_i.
\end{equation}

Direct multiplication gives
\[
 (C^{-1})_{ij}
 =\frac{\mathbf1_{\{i=j\}}-\mathbf1_{\{i=j+1\}}}{\sqrt{q_ip_j}},
 \qquad
 W^{-1}=
 \begin{pmatrix}
  {\bf 0}&(C^{\mathsf T})^{-1}\\
  C^{-1}&{\bf 0}
 \end{pmatrix}.
\]
Let  $\mathbf e_i$ denote the $i$-th
standard basis vector.  Put
$\mathbf y= \begin{pmatrix}
  \mathbf e_k/\sqrt{p_k}\\
  \mathbf e_1/\sqrt{q_1}
 \end{pmatrix}.
$ Then
$W\mathbf y=\mathbf x$ and $\mathbf x^{\mathsf T}\mathbf y=2$, so
direct multiplication yields
\[
 M_c^{-1}=
 \begin{pmatrix}
  -W^{-1}&\mathbf y/\sqrt c\\
  \mathbf y^{\mathsf T}/\sqrt c&-1/c
 \end{pmatrix}.
\]

Let the off-diagonal support graph of $M_c^{-1}$ have the coordinates
of $M_c^{-1}$ as its vertices, with two distinct coordinates adjacent
whenever the corresponding matrix entry is nonzero.  The formula for
$C^{-1}$ shows that the $j$-th row coordinate is adjacent to the
$j$-th column coordinate and, when $j<k$, to the $(j+1)$-st column
coordinate.  The block $\mathbf y/\sqrt c$ gives the two additional
edges joining the isolated coordinate to the first column coordinate
and the $k$-th row coordinate.  Thus the graph is a $(2k+1)$-cycle,
with the coordinates occurring in the cyclic order: isolated, first
column, first row, second column, second row, and so on.  The only
nonzero diagonal entry of $M_c^{-1}$ is $-1/c$ at the isolated
coordinate.

Let $M=(m_{ij})_{1\leq i,j\leq s}$ be a proper principal submatrix of
$M_c^{-1}$.  Thus $1\leq s\leq2k$, and its off-diagonal support graph
is a disjoint union of paths.  The Leibniz expansion is
\[
 \det M
 =\sum_{\sigma\in\mathfrak S_s}
   \operatorname{sgn}(\sigma)\prod_{i=1}^s m_{i,\sigma(i)}.
\]
If a term in this sum is nonzero, then
$m_{i,\sigma(i)}\ne0$ for every $i$.  A cycle of length at least three
in $\sigma$ would therefore give a cycle in the off-diagonal support
graph of $M$, which is impossible.  Hence $\sigma$ consists only of
disjoint transpositions and fixed points.  Since the only nonzero
diagonal entry of $M_c^{-1}$ is $-1/c$ at the isolated coordinate,
this coordinate is the only possible fixed point.

If $s=2r$, every nonzero term consists of $r$ transpositions.  For
each transposition $(i\,j)$, its permutation sign is $-1$, while
$m_{ij}m_{ji}=m_{ij}^2>0$ because $M$ is symmetric.  Thus every
nonzero term has sign $(-1)^r$.  If $s=2r+1$, every nonzero term
consists of $r$ transpositions together with the fixed isolated
coordinate, whose diagonal entry is $-1/c$.  Hence every such term
has sign $(-1)^{r+1}$.  For a fixed principal submatrix of order $2r$, its determinant is
either zero or a sum of nonzero terms all having sign $(-1)^r$.
Hence, if the determinant is nonzero, it has sign $(-1)^r$.
Similarly, a principal minor of order $2r+1$ is either zero or has
sign $(-1)^{r+1}$.

For every $1\leq s\leq2k$, there is a nonzero principal minor of
order $s$.  If $s=2r$, choose $r$ pairwise disjoint edges of the
cycle and take their $2r$ endpoints.  The permutation that exchanges
the two endpoints of each chosen edge contributes a nonzero term to
the corresponding determinant.  If $s=2r+1$, take the isolated
coordinate together with the endpoints of $r$ pairwise disjoint edges
in the path obtained by deleting that coordinate.  The permutation
that fixes the isolated coordinate and exchanges the endpoints of
each chosen edge again contributes a nonzero term.  Since all nonzero
terms in either determinant have the same sign, the resulting
principal minors are nonzero.

Write $E(t)=\det(I+tM_c)=\sum_{j=0}^{2k+1}E_jt^j$. For a set $S$ of coordinates, let $M_c^{-1}[S]$ denote the principal
submatrix indexed by $S$.  In the expansion of
$\det(M_c^{-1}+tI)$, choosing $t$ from $j$ diagonal positions leaves
the determinant of the principal submatrix on the remaining
$2k+1-j$ coordinates.  Hence the identity $E(t)=\det(M_c)\det(M_c^{-1}+tI)$ shows
\[
 E_j
 =\det(M_c)
  \sum_{\substack{
        S\subseteq\{1,\ldots,2k+1\}\\
        |S|=2k+1-j}}
  \det\bigl(M_c^{-1}[S]\bigr).
\]
For each proper positive order, the preceding argument shows that the
sum on the right is nonzero.  It has sign $(-1)^r$ when its order is
$2r$, and sign $(-1)^{r+1}$ when its order is $2r+1$.  Moreover,
\eqref{eq:ferrers-det} gives
$\operatorname{sgn}\det M_c=(-1)^k$.
For $j=2\ell+1$, the relevant principal minors have order
$2(k-\ell)$, while for $j=2\ell$ they have order
$2(k-\ell)+1$.  Therefore
\[
 \begin{aligned}
  \operatorname{sgn}E_{2\ell+1}
  &=(-1)^k(-1)^{k-\ell}
    =(-1)^\ell
    &&(0\leq\ell\leq k),\\
  \operatorname{sgn}E_{2\ell}
  &=(-1)^k(-1)^{k-\ell+1}
    =(-1)^{\ell-1}
    &&(1\leq\ell\leq k).
 \end{aligned}
\]
Here the principal minor of order zero is understood to be $1$.
Together with $E_0=1$, this shows that every $E_j$ is nonzero.  The
signs of $E_0,E_1,\ldots,E_{2k+1}$ are
$+,+,+,-,-,+,+,-,-,\ldots$, so the sign changes occur exactly in
positions
\begin{equation}\label{eq:E-change-positions}
 3,5,7,\ldots,2k+1.
\end{equation}
Set
\[
 F(t)=\prod_{i=3}^{2k+1}(1+\lambda_it)
     =\sum_{j=0}^{2k-1}f_jt^j.
\]
Then $ E(t)=(1+\lambda_1t)(1+\lambda_2t)F(t).$
Since $M_c$ has inertia $(k+1,k,0)$, both $\lambda_1$ and
$\lambda_2$ are positive.
Define
\[
 \mathbf u=
 \frac{
  (\sqrt{p_1},\ldots,\sqrt{p_k},
   \overbrace{0,\ldots,0}^{k},\sqrt c)^{\mathsf T}}
 {\sqrt{c+\sum_{i=1}^kp_i}},
 \qquad
 \mathbf v=
 \frac{
  (\overbrace{0,\ldots,0}^{k},
   \sqrt{q_1},\ldots,\sqrt{q_k},0)^{\mathsf T}}
 {\sqrt{\sum_{i=1}^kq_i}}.
\]
Their supports are disjoint and both have norm one, so they are
orthonormal.  The principal block of $M_c$ on the row coordinates
together with the isolated coordinate is the outer product of
$(\sqrt{p_1},\ldots,\sqrt{p_k},\sqrt c)^{\mathsf T}$ with itself.
The principal block on the column coordinates is the outer product of
$(\sqrt{q_1},\ldots,\sqrt{q_k})^{\mathsf T}$ with itself.  Hence
\[
 \mathbf u^{\mathsf T}M_c\mathbf u
 =c+\sum_{i=1}^kp_i,
 \qquad
 \mathbf v^{\mathsf T}M_c\mathbf v
 =\sum_{i=1}^kq_i.
\]
Since
$\operatorname{tr}M_c=c+\sum_i p_i+\sum_iq_i$,
Lemma~\ref{lem:ky-fan} gives
$
 \lambda_1+\lambda_2
 \geq
 \mathbf u^{\mathsf T}M_c\mathbf u+
 \mathbf v^{\mathsf T}M_c\mathbf v
 =\operatorname{tr}M_c.$
Consequently,
\[
 f_1=\sum_{i=3}^{2k+1}\lambda_i
    =\operatorname{tr}M_c-\lambda_1-\lambda_2
    \leq0.
\]
The polynomial $F$ has $k$ positive roots and $k-1$ negative roots,
$(1+\lambda_2t)F$ has $k$ roots of each sign, and $E$ has $k$
positive roots and $k+1$ negative roots.

 Suppose that $f_2>0$.  We already know that $f_1\leq0$.
If $f_1=0$, then
$2f_2=f_1^2-\sum_{i=3}^{2k+1}\lambda_i^2<0$, because $M_c$ is nonsingular.
This contradicts $f_2>0$, so $f_1<0$.  Since $f_0=1$, the coefficient
sequence of $F$ begins with $+,-,+$.  Its first two sign changes
therefore occur in positions $1$ and $2$.

To remove possible zero coefficients, regard $
 (\lambda_1,\ldots,\lambda_{2k+1})$
temporarily as an independent tuple of real variables.  Every
coefficient of $F$, $(1+\lambda_2t)F$, and $E$ is a nonzero
polynomial in this tuple.  Hence the product of these finitely many
coefficient polynomials is a nonzero polynomial, and its zero set has
empty interior.  At the original tuple, all $\lambda_i$ and all
coefficients $E_j$ are nonzero, while $f_1<0<f_2$.  These strict sign
conditions persist in an open neighborhood of the original tuple.
We may therefore choose a tuple in this neighborhood for which none
of the three coefficient sequences contains a zero term.  At this
tuple, every $\lambda_i$ and every $E_j$ has its original sign, and
$f_1<0<f_2$ still holds.  For notational simplicity, retain
$\lambda_i$, $F$, and $E$ for the perturbed tuple.

Thus $F$, $(1+\lambda_2t)F$, and $E$ each have $k$ positive roots.
Lemma~\ref{lem:real-rooted-variation} shows that each has exactly $k$
coefficient sign changes.  The second sign change of $F$ is in
position $2$.  Lemma~\ref{lem:variation-location} shows that after
multiplication by $1+\lambda_2t$ it is in position $2$ or $3$, and
after the subsequent multiplication by $1+\lambda_1t$ it is in a
position no larger than $4$.  This contradicts
\eqref{eq:E-change-positions}, where the second sign change of $E$ is
in position $5$.  Therefore
$ \sum_{3\leq i<j\leq2k+1}\lambda_i\lambda_j=f_2\leq0.$

Finally, $W^{-1}\mathbf x=\mathbf y$ and
$\mathbf x^{\mathsf T}\mathbf y=2$, so
$\mathbf x^{\mathsf T}W^{-1}\mathbf x=2$.  Consider the bordered matrix
\[
 \widetilde M=
 \begin{pmatrix}
  -W&\mathbf x\\
  \mathbf x^{\mathsf T}&-1
 \end{pmatrix}.
\]
The Schur complement of $-W$ in $\widetilde M$ is
$-1-\mathbf x^{\mathsf T}(-W)^{-1}\mathbf x=1$.
Since $-W$ has inertia $(k,k,0)$, it follows that $\widetilde M$ has
inertia $(k+1,k,0)$.  On the other hand, the Schur complement of $-1$ in $ \widetilde M$  is $
 -W-\mathbf x(-1)^{-1}\mathbf x^{\mathsf T}=M_0.$
Therefore the inertia of $\widetilde M$ is the inertia of $M_0$ plus
$(0,1,0)$, which yields the desired result.
\end{proof}

We now transfer Lemma~\ref{lem:positive-ferrers-tail} from weighted
Ferrers quotients to ordinary chain graphs.  The case with no isolated
vertices follows by a limiting argument.

\begin{theorem}\label{thm:ferrers-tail}
Let $H$ be a bipartite chain graph on $n$ vertices.  If
$\mu_1\geq\cdots\geq\mu_n$ are the eigenvalues of $J-A(H)$, then
$ \sum_{3\leq i<j\leq n}\mu_i\mu_j\leq0.$
\end{theorem}

\begin{proof}
Let $k$ be the incidence rank of $H$, and let $N$ be a bipartite
incidence matrix of $H$.  After ordering the vertices according to the
two parts,
\begin{equation}\label{ah}
 A(H)=
 \begin{pmatrix}
  {\bf 0}&N\\
  N^{\mathsf T}&{\bf 0}
 \end{pmatrix}.
\end{equation}
Multiplying all rows and columns indexed by the second part of the
bipartition by $-1$ leaves the determinant unchanged and transforms
$tI-A(H)$ into $tI+A(H)$.  Hence
\[
 \det(tI-A(H))
 =\det(tI+A(H))
 =(-1)^n\det(-tI-A(H)).
\]
Hence the nonzero eigenvalues of $A(H)$ occur in opposite pairs.
Since
$\operatorname{rank}A(H)=2\operatorname{rank}N=2k$, there are exactly
$k$ such pairs.  In particular, if $k\leq1$, there is at most one.
Thus $-A(H)$ has at most one positive and at most one negative
eigenvalue.  Since $J$ is positive semidefinite of rank one,
rank-one interlacing shows that $J-A(H)$ has at most two positive and
at most one negative eigenvalue.  Hence at most one of $\mu_3,\ldots,\mu_n$ is nonzero, so $\sum_{3\leq i<j\leq n}\mu_i\mu_j=0.$
This proves the asserted inequality when $k\leq1$.

Assume now that $k\geq2$, and let $c$ be the number of isolated
vertices of $H$.  If $c>0$, Lemma~\ref{lem:positive-ferrers-tail}
gives the required inequality for $M_c$.  Since $M_c$ has inertia
$(k+1,k,0)$, Lemma~\ref{lem:ferrers-quotient} shows that the tail
$\mu_3,\ldots,\mu_n$ consists of the tail of $M_c$ together with
zeros.  Hence the asserted inequality follows.
Suppose that $c=0$.  For $\varepsilon>0$, put
\[
 M_\varepsilon=
 \begin{pmatrix}
  M_0&\sqrt{\varepsilon}\,\mathbf x\\
  \sqrt{\varepsilon}\,\mathbf x^{\mathsf T}&\varepsilon
 \end{pmatrix}.
\]
Lemma~\ref{lem:positive-ferrers-tail} applies to $M_\varepsilon$, and $
 \lim_{\varepsilon\to0^+}M_\varepsilon=M_0\oplus[{\bf 0}].$
The ordered eigenvalues of a real symmetric matrix depend continuously
on its entries, so the required inequality passes to
$M_0\oplus[{\bf 0}]$.  Since $k\geq2$ and $M_0$ has inertia $(k+1,k-1,0)$, $M_0$ has at
least three positive eigenvalues.  Thus adjoining a zero changes
neither its two largest eigenvalues nor $
 \sum_{3\leq i<j}\lambda_i\lambda_j,$
because every new product involving the adjoined eigenvalue is zero.  Lemma~\ref{lem:ferrers-quotient} therefore gives the required
inequality for $J-A(H)$.
\end{proof}

The tail inequality controls all pairwise products below the two
largest eigenvalues.  A nonsingular binary $2\times2$ minor now turns
that qualitative control into a numerical gap uniform over all chain
graphs of incidence rank at least two.

For a real matrix $Q$ of rank $r$, write $\sigma_1(Q)\geq\sigma_2(Q)\geq\cdots\geq\sigma_r(Q)>0$
for the positive square roots of the nonzero eigenvalues of
$Q^{\mathsf T}Q$, listed with multiplicity in decreasing order.  These
are the nonzero singular values of $Q$.
Let $H$ be a bipartite graph on $n$ vertices, and let
$N$ be a bipartite incidence matrix of $H$.  Put $k=\rank N$, the incidence rank of $H$.
For brevity, write $ \sigma_i=\sigma_i(N)\ (1\leq i\leq k).$
By \eqref{ah}, the nonzero eigenvalues of $A(H)$ are
$\pm\sigma_1,\ldots,\pm\sigma_k$.

\begin{lemma}\label{lem:defect}
Let  $H$ be  a chain graph of incidence rank at least two and  put $M=J-A(H)$. Then
\[
 2\bigl(
 \lambda_1(M)\lambda_2(M)-\lambda_2(M)^2
 +S_2(M)\bigl(n-S_2(M)\bigr)
 \bigr)>\frac12.
\]
\end{lemma}

\begin{proof}
Write $e(H)=|E(H)|$.  Since $N$ has exactly $e(H)$ entries equal to
one and $M=J-A(H)$ has $n^2-2e(H)$ entries equal to one,
\[
 e(H)=\tr(N^{\mathsf T}N)=\sum_{i=1}^k\sigma_i^2,
 \qquad
 \tr M=n,
 \qquad
 \tr(M^2)=\sum_{u,v}M_{uv}^2=n^2-2e(H).
\]
Therefore
\[
 \sum_{1\leq i<j\leq n}\lambda_i(M)\lambda_j(M)
 =\frac{(\tr M)^2-\tr(M^2)}2=e(H).
\]
Since $\sum_{i=3}^n\lambda_i(M)=n-S_2(M),$
separating the pairs involving the two largest eigenvalues gives
\[
\begin{aligned}
 e(H)=\lambda_1(M)\lambda_2(M)
   +S_2(M)\bigl(n-S_2(M)\bigr)+\sum_{3\leq i<j\leq n}
          \lambda_i(M)\lambda_j(M).
\end{aligned}
\]
Consequently,
\[
\begin{aligned}
 &\lambda_1(M)\lambda_2(M)-\lambda_2(M)^2
   +S_2(M)\bigl(n-S_2(M)\bigr)\\
 &\qquad=(\sigma_1^2-\lambda_2(M)^2)
   +\sum_{i=2}^k\sigma_i^2
   -\sum_{3\leq i<j\leq n}
      \lambda_i(M)\lambda_j(M).
\end{aligned}
\]

The positive eigenvalues of $-A(H)$ begin with
$\sigma_1,\sigma_2$.  Since $M=-A(H)+J$ and $J$ is positive
semidefinite of rank one, rank-one interlacing gives
$ \lambda_1(M)\geq\sigma_1
 \geq\lambda_2(M)\geq\sigma_2>0.$
Together with Theorem~\ref{thm:ferrers-tail}, this implies
\begin{equation}\label{a1}
\lambda_1(M)\lambda_2(M)-\lambda_2(M)^2
 +S_2(M)\bigl(n-S_2(M)\bigr)
 \geq\sigma_2^2.
\end{equation}
By hypothesis, $\rank N\geq2$, so $N$ contains a nonsingular
$2\times2$ submatrix $R$.  Up to row and column permutations, $R$ is
either $I_2$ or $
 \begin{pmatrix}
  1&1\\
  0&1
 \end{pmatrix}.$
The first has both squared singular values equal to one.  For the
second,
\[
 R^{\mathsf T}R=
 \begin{pmatrix}
  1&1\\
  1&2
 \end{pmatrix},
 \qquad
 \det(xI-R^{\mathsf T}R)=x^2-3x+1.
\]
Thus $\sigma_2(R)^2$ is either $1$ or $(3-\sqrt5)/2$, and hence

\begin{equation}\label{a2}
\sigma_2(R)^2\geq\frac{3-\sqrt5}{2}.
\end{equation}
Let $N_0$ consist of the two columns of $N$ containing $R$.  After
permuting rows and columns, write
\[
 N_0=\begin{pmatrix}R\\ X\end{pmatrix},
 \qquad
 N=\begin{pmatrix}N_0&Y\end{pmatrix}.
\]
Then
\[
 N_0^{\mathsf T}N_0
 =R^{\mathsf T}R+X^{\mathsf T}X,
 \qquad
 NN^{\mathsf T}
 =N_0N_0^{\mathsf T}+YY^{\mathsf T}.
\]
Both $X^{\mathsf T}X$ and $YY^{\mathsf T}$ are positive
semidefinite.  Eigenvalue monotonicity, together with the equality of
the nonzero spectra of $N^{\mathsf T}N$ and $NN^{\mathsf T}$ and
similarly for $N_0$, gives
\begin{equation}\label{a3}
 \sigma_2^2\ge \sigma_2(N_0)^2\ge \sigma_2(R)^2.
\end{equation}
Consequently, \eqref{a1}, \eqref{a2} and \eqref{a3} yield
\[
\begin{aligned}
 2\bigl(
 \lambda_1(M)\lambda_2(M)-\lambda_2(M)^2
 +S_2(M)\bigl(n-S_2(M)\bigr)
 \bigr)
 \geq3-\sqrt5>\frac12.
\end{aligned}
\]
\end{proof}

We now combine the preceding estimates with the structural reduction.
The defect bound first forces an edge-maximal maximizer into incidence
rank one, where the integer optimization determines its value.  We
then return to an arbitrary maximizer and analyze equality in the
variational threshold to obtain uniqueness.

\begin{proof}[\textbf{\upshape Proof of Theorem~\ref{thm:main}}]
We first determine the maximum value.  Choose an $n$-vertex graph
$G$ maximizing $S_2(G)$ and, among all such maximizers, choose one
with the largest number of edges.  By
Proposition~\ref{prop:ferrers-reduction}, $G$ is connected and
$\overline G$ is a chain graph.  Put
\[
 a=\lambda_1(B(G)),
 \qquad
 b=\lambda_2(B(G)),
\]
and let $k$ be the incidence rank of $\overline G$.  Recall that
$-\tau_n$ is the negative eigenvalue of $B(K_n^\star)$.  By the maximality of $G$, \eqref{eq:candidate-value} and
\eqref{eq:tau-envelope} give
\[
 a+b=S_2(B(G))
 \geq S_2(B(K_n^\star))
 =S_2(K_n^\star)+2
 =n+\tau_n>n.
\]

If $k=0$, then $G=K_n$ and $a+b=S_2(J)=n$, contradicting the
preceding lower bound. If  $k\geq2$, then put
$ D=2\bigl(ab-b^2+(a+b)(n-a-b)\bigr).$
Since $B(G)=J-A(\overline G)$, Lemma~\ref{lem:defect} gives
$D>1/2$.  Using the definition of $f_n$ and expanding $D$, we obtain
\[
\begin{aligned}
 f_n(a+b)
 &=\frac34(a+b)^2+(a+b-n)^2\\
 &=n^2-D-\frac14(a-3b)^2\\
 &<n^2-\frac12.
\end{aligned}
\]
This together with the upper bound in \eqref{eq:candidate-deficit} gives
$f_n(a+b)<f_n(n+\tau_n).$
However,
\[
 f_n'(x)=\frac72x-2n>0
 \qquad(x>4n/7).
\]
Since $a+b\geq n+\tau_n>n$, it follows that
$f_n(a+b)\geq f_n(n+\tau_n)$, a contradiction.  Thus $k\geq2$ is
also impossible.

Therefore $k=1$.  Hence $\overline G$ consists of a complete
bipartite graph and isolated vertices, so
$G\cong K(n,p,q)$ for some $p,q>0$ with $p+q\leq n$.
Proposition~\ref{prop:integer-opt} and
\eqref{eq:candidate-value} give
\[
 a+b=S_2(B(G))
     =S_2(G)+2
     \leq S_2(K_n^\star)+2
     =n+\tau_n.
\]
Together with the preceding lower bound, this yields
$a+b=n+\tau_n$. It follows that
\begin{equation}\label{eq:looped-maximum}
 \max_{|V(G)|=n}S_2(B(G))=n+\tau_n.
\end{equation}
Together with \eqref{eq:candidate-value}, this proves
\eqref{eq:main-bound}.

We now determine the equality case.  Let $G$ be any $n$-vertex graph
attaining equality in \eqref{eq:main-bound}.  We first show that $G$
is connected.  Suppose
otherwise.  If its two largest eigenvalues arise from distinct
components of orders $n_1$ and $n_2$, then they are the spectral
radii of those components, and hence
\[
 S_2(G)\leq(n_1-1)+(n_2-1)\leq n-2
 <S_2(K_n^\star),
\]
a contradiction.
Thus both eigenvalues arise from one component $H$ of order $m<n$,
and $S_2(G)=S_2(H)$. If $m\geq5$, applying \eqref{eq:main-bound} at order $m$ and then Lemma~\ref{lem:strict-growth}
gives $S_2(H)\leq S_2(K_m^\star)<S_2(K_n^\star),$
again a contradiction.

It remains to consider $m\leq4$.  Since $H$ is connected and
contributes two eigenvalues, $2\leq m\leq4$.  Write $K_3^+$ for the graph obtained from $K_3$ by adjoining a pendant vertex.  Up to isomorphism, the connected graphs of orders
$2$, $3$, and $4$ are exactly those listed below, and direct
calculation gives
\[
\setlength{\arraycolsep}{4pt}
\renewcommand{\arraystretch}{1.25}
\begin{array}{c|c|cc|cccccc}
 |V(H)|
 &2
 &\multicolumn{2}{c|}{3}
 &\multicolumn{6}{c}{4}\\ \hline
 H
 &K_2
 &P_3&K_3
 &P_4&K_{1,3}&C_4&K_3^+&K_4-e&K_4\\
 S_2(H)
 &0
 &\sqrt2&1
 &\sqrt5&\sqrt3&2&1-\gamma&
   \dfrac{1+\sqrt{17}}2&2
\end{array}
\]
where $\gamma$ denotes a negative root of $x^3-x^2-3x+1$.
Descartes' rule applied after replacing $x$ by $-x$ shows that it is
the unique negative root.  Since the polynomial is negative at
$-2$ and positive at $-1$, we have
$-2<\gamma<-1$, and hence $
 1-\gamma<3.$
All the other entries in the table are also less than $3$.  Therefore
$S_2(H)<3$.
By \eqref{eq:candidate-value} and
Lemma~\ref{lem:strict-growth},
$ S_2(K_n^\star)\geq S_2(K_5^\star)
 >3>S_2(H)=S_2(G),$
contrary to $S_2(G)=S_2(K_n^\star)$.  Thus $G$ is connected.

Consequently, $B(G)$ is irreducible.  Choose a positive unit Perron
eigenvector $\mathbf x$ of $B(G)$ and a unit eigenvector
$\mathbf y\perp\mathbf x$ for its second largest eigenvalue, and set
$\kappa_{uv}=x_ux_v+y_uy_v.$
For $u\ne v$, adding the edge $uv$ changes
$
 \mathbf x^{\mathsf T}B(G)\mathbf x+
 \mathbf y^{\mathsf T}B(G)\mathbf y$
by $2\kappa_{uv}$, while deleting it changes this sum by
$-2\kappa_{uv}$.  Lemma~\ref{lem:ky-fan} and the optimality of $G$
therefore give
\begin{equation}\label{eq:weak-threshold}
 \kappa_{uv}\leq0\quad\text{if }B(G)_{uv}=0,
 \qquad
 \kappa_{uv}\geq0\quad\text{if }B(G)_{uv}=1
 \qquad(u\ne v).
\end{equation}

Define $\widehat B$ by $\widehat B_{uu}=1$ and, for $u\ne v$,
$
 \widehat B_{uv}=
 \begin{cases}
  1,&\kappa_{uv}\geq0,\\
  0,&\kappa_{uv}<0.
 \end{cases}
$
Equation~\eqref{eq:weak-threshold} shows that every one-entry of
$B(G)$ remains one in $\widehat B$.  For a zero entry, it gives
$\kappa_{uv}\leq0$, so the definition of $\widehat B$ changes that
entry to one exactly when $\kappa_{uv}=0$.  Each such change contributes
$2\kappa_{uv}=0$ to the Rayleigh sum.  Therefore
\[
 \mathbf x^{\mathsf T}\widehat B\mathbf x
 +\mathbf y^{\mathsf T}\widehat B\mathbf y
 =S_2(B(G))=n+\tau_n.
\]

By \eqref{eq:looped-maximum} and Lemma~\ref{lem:ky-fan},
\[
 n+\tau_n\geq S_2(\widehat B)
 \geq\mathbf x^{\mathsf T}\widehat B\mathbf x+
      \mathbf y^{\mathsf T}\widehat B\mathbf y
 =n+\tau_n.
\]
Thus equality holds throughout, and
$\mathcal U=\operatorname{span}\{\mathbf x,\mathbf y\}$ attains the
maximum in Lemma~\ref{lem:ky-fan} for $\widehat B$.

Since $\widehat B\geq B(G)$ entrywise and $B(G)$ is irreducible,
$\widehat B$ is irreducible.  Its Perron eigenvalue is therefore
simple.  The equality statement in Lemma~\ref{lem:ky-fan} shows that
$\mathcal U$ has an orthonormal basis consisting of a positive Perron
eigenvector $\widehat{\mathbf x}$ of $\widehat B$ and an eigenvector
$\widehat{\mathbf y}$ for its second largest eigenvalue.  Since both
pairs are orthonormal bases of $\mathcal U$,
\[
 \widehat{\mathbf x}\widehat{\mathbf x}^{\mathsf T}
 +\widehat{\mathbf y}\widehat{\mathbf y}^{\mathsf T}
 =
 \mathbf x\mathbf x^{\mathsf T}
 +\mathbf y\mathbf y^{\mathsf T}.
\]
In particular, $\widehat x_u\widehat x_v+\widehat y_u\widehat y_v
 =\kappa_{uv}.$
The definition of $\widehat B$ and the preceding identity satisfy the
hypotheses of Lemma~\ref{lem:threshold-complement}.  Hence the graph
with adjacency matrix $J-\widehat B$ is a chain graph.

Let $\widehat H$ be the chain graph with adjacency matrix
$J-\widehat B$.  If $\widehat H$ had incidence rank zero, then its
bipartite incidence matrix, and hence $A(\widehat H)$, would be zero.
Thus $\widehat B=J$, giving
$S_2(\widehat B)=n<n+\tau_n$, a contradiction.
If $\widehat H$ had incidence rank at least two, the preceding argument
excluding $k\geq2$, applied to
$\widehat B=J-A(\widehat H)$, would give
$ f_n(S_2(\widehat B))<n^2-\frac12.$
On the other hand, $S_2(\widehat B)=n+\tau_n$ and
\eqref{eq:candidate-deficit} give
\[
 f_n(S_2(\widehat B))
 =f_n(n+\tau_n)>n^2-\frac12,
\]
again a contradiction.  Hence $\widehat H$ has incidence rank one.
As observed after Proposition~\ref{prop:ferrers-reduction}, an
incidence-rank-one chain graph consists of a complete bipartite graph
and isolated vertices.  Hence
$\widehat B=B(K(n,p,q))$ for some $p,q>0$ with $p+q\leq n$.
Since $\widehat B$ is a global maximizer,
Proposition~\ref{prop:integer-opt} gives, after relabelling,
$ \widehat B=B(K_n^\star).$

It remains to show that $\widehat B=B(G)$.  Set
\[
 p=\left\lceil\frac{s_n}{2}\right\rceil,\qquad
 q=\left\lfloor\frac{s_n}{2}\right\rfloor,\qquad
 c=n-s_n.
\]
Since the rows of $\widehat B=B(K_n^\star)$ are constant on each of
the $p$-, $q$-, and $c$-parts, the image of $\widehat B$ is contained
in the three-dimensional space of vectors that are constant on these
parts.  Lemma~\ref{lem:candidate-poly} gives
$\operatorname{rank}\widehat B=3$, so its image is exactly this space.

The same lemma shows that $\widehat B$ has two positive eigenvalues
and one negative eigenvalue $-\tau_n$.  Since $\mathcal U$ is a
maximizing two-dimensional subspace for $\widehat B$, it is the
positive spectral subspace.  Let $\mathbf w$ be a unit eigenvector for
$-\tau_n$.  Since
$\mathbf w=-\tau_n^{-1}\widehat B\mathbf w$, it belongs to the image
of $\widehat B$ and is therefore constant on each part.  Denote its
values on the $p$-, $q$-, and $c$-parts by $w_p,w_q,w_c$,
respectively.  The first two coordinates of
$\widehat B\mathbf w=-\tau_n\mathbf w$ give
\[
 (p+\tau_n)w_p+cw_c=0,
 \qquad
 (q+\tau_n)w_q+cw_c=0.
\]
Since $p,q,c,\tau_n>0$, these equations show that
$w_p,w_q,w_c$ are all nonzero.

The vectors $\mathbf x,\mathbf y,\mathbf w$ form an orthonormal basis
of the space of vectors that are constant on the three parts.  The
normalized characteristic vectors of the three parts form another
orthonormal basis of the same space.  Since the sum of the outer
products of an orthonormal basis depends only on the subspace, for
distinct vertices within the $p$-, $q$-, and $c$-parts, respectively,
\[
 \kappa_{uv}=\frac1p-w_p^2,\qquad
 \kappa_{uv}=\frac1q-w_q^2,\qquad
 \kappa_{uv}=\frac1c-w_c^2.
\]
Moreover,
$pw_p^2+qw_q^2+cw_c^2=1$.  Since all three terms are positive, each
of the preceding three quantities is positive.

For vertices in the $p$- and $c$-parts, the same comparison of the two
orthonormal bases and the first eigenvector equation give
$ \kappa_{uv}=-w_pw_c
 =\frac{c\,w_c^2}{p+\tau_n}>0.$
Similarly, between the $q$- and $c$-parts,
$ \kappa_{uv}=-w_qw_c
 =\frac{c\,w_c^2}{q+\tau_n}>0.$
These are exactly the off-diagonal positions at which $\widehat B$
has entry one.  Hence $\kappa_{uv}>0$
whenever $u\ne v$ and $\widehat B_{uv}=1.$

We have already shown that every one-entry of $B(G)$ remains one in
$\widehat B$.  Conversely, let $u\ne v$ and suppose that
$\widehat B_{uv}=1$.  Then $\kappa_{uv}>0$ by the preceding
inequality.  If $B(G)_{uv}=0$, however,
\eqref{eq:weak-threshold} would give $\kappa_{uv}\leq0$, a
contradiction.  Hence $B(G)_{uv}=1$.  Thus $B(G)$ and $\widehat B$
have the same off-diagonal one-entries.  Their diagonal entries are
also equal to one, so $ B(G)=\widehat B.$
Consequently, $G\cong K_n^\star$.  Conversely,
\eqref{eq:candidate-value} shows that $K_n^\star$ attains equality.
Finally, \eqref{eq:candidate-value} and \eqref{eq:tau-envelope} give
\[
 S_2(K_n^\star)=n-2+\tau_n\leq\frac{8n}{7}-2,
\]
with equality if and only if $7\mid n$.
This completes the proof.

\end{proof}

\end{document}